\documentclass[12pt]{amsart}

\usepackage{amsmath,amsfonts,amsthm,amssymb,fullpage,color,enumerate}

\usepackage{xcolor}
\usepackage[backref=page]{hyperref}
\definecolor{couleur_cite}{rgb}{0.05,.4,0.05}
\definecolor{couleur_link}{rgb}{0.05,0.05,0.4}
\hypersetup{
bookmarksopen,bookmarksnumbered,
colorlinks,
linkcolor=couleur_link,citecolor=couleur_cite,
}

\newtheorem{theorem}{Theorem}
\newtheorem{lemma}[theorem]{Lemma}
\newtheorem{prop}[theorem]{Proposition}
\newtheorem{problem}{Problem}
\newtheorem{cor}[theorem]{Corollary}

\theoremstyle{remark}

\newcommand{\R}{\mathbb R}
\newcommand{\C}{\mathbb C}

\newcommand{\Z}{\mathbb Z}
\newcommand{\Q}{\mathbb Q}
\newcommand{\F}{\mathbb F}
\newcommand{\A}{\mathbb A}

\newcommand{\g}{\mathfrak g}

\newcommand{\cV}{\mathcal V}
\newcommand{\cA}{\mathcal A}

\newcommand{\cP}{\mathcal P}

\newcommand{\Hom}{\text{Hom}}

\newcommand{\be}{\begin{equation}}
\newcommand{\ee}{\end{equation}}
\newcommand{\bes}{\begin{equation*}}
\newcommand{\ees}{\end{equation*}}

\newcommand{\ba}{\begin{eqnarray}}
\newcommand{\ea}{\end{eqnarray}}
\newcommand{\bas}{\begin{eqnarray*}}
\newcommand{\eas}{\end{eqnarray*}}

\begin{document}

\title{Bounds for the number of cohomological automorphic representations of $GL_3/\Q$ in the weight aspect}
\author{Simon Marshall}
\address{Department of Mathematics\\
University of Wisconsin -- Madison\\
480 Lincoln Drive\\
Madison\\
WI 53706, USA}
\email{marshall@math.wisc.edu}

\begin{abstract}

We prove a power saving over the trivial bound for the number of cohomological cuspidal automorphic representations of fixed level and growing weight on $GL_3/\Q$, by adapting the methods of our earlier paper on $GL_2$.

\end{abstract}

\maketitle

\section{Introduction}

The purpose of this article is to prove a power saving over the trivial bound for the number of cohomological cuspidal automorphic representations on $GL_3/\Q$ of fixed level and growing weight.
  To state our result, let $a \ge b \ge c$ be integers, and let $V$ be the irreducible algebraic representation of $GL_3(\R)$ with highest weight $(a,b,c)$.  By a theorem of Borel and Wallach \cite[Ch II, Prop 6.12]{BW}, if an irreducible unitary representation $\pi$ of $GL_3(\R)$ has nonzero $(\g,K)$-cohomology with coefficients in $V$, then $b = 0$ and $c = -a$ (i.e. $V$ is equivalent to its twist by the Cartan involution).  We shall therefore restrict our attention to the coefficient systems $V_\lambda$ with highest weight $(\lambda, 0, -\lambda)$, and say that an irreducible unitary representation $\pi$ is cohomological of weight $\lambda$ if it is infinite dimensional and $H^*(\g,K; \pi \otimes V_\lambda) \neq 0$.  (Note that we take $K = SO(3)$.)  It is known that there are two such $\pi$, which are trivial on the positive scalar matrices and are twists of each other by the sign of the determinant.  Moreover, they satisfy
\bes
H^i(\g,K; \pi \otimes V_\lambda) = \Big\{ \begin{array}{ll} \C & i = 2,3, \\ 0 & i \neq 2,3. \end{array}
\ees
We shall say that an automorphic representation $\pi$ on $GL_3/\Q$ is cohomological of weight $\lambda$ if $\pi_\infty$ has this property.  Our main result is the following:

\begin{theorem}
\label{maincusp}

Fix a compact open subgroup $K \subset GL_3(\A_f)$, and let $\cA_\lambda$ be the set of cuspidal automorphic representations on $GL_3/\Q$ that are cohomological of weight $\lambda$ and have level $K$.  We have $| \cA_\lambda| \ll_{K,\epsilon} \lambda^{3 - 4/27 + \epsilon}$.

\end{theorem}

We shall deduce Theorem \ref{maincusp} from the following theorem on the cohomology of congruence subgroups of $SL(3,\Z)$.

\begin{theorem}
\label{maincoh}

Let $\Gamma$ be a congruence subgroup of $SL(3,\Z)$.  We have $\dim H^2(\Gamma, V_\lambda) \ll_{\Gamma, \epsilon} \lambda^{3 - 4/27 + \epsilon}$.

\end{theorem}

We note that the trivial bound in Theorems \ref{maincusp} and \ref{maincoh} is on the order of $\dim V_\lambda \sim \lambda^3$.  As a result, these theorems represent a power improvement for the dimension of a space of automorphic forms that are tempered but not essentially square integrable at infinity.  This is a difficult problem, which has only been solved in a few cases \cite{CE1,D,Hu,Ma,MV}.  (See also the paper \cite{Sr} of Sardari for an analogous result at finite places.)  Moreover, there are currently no results of this type proved using the trace formula alone, despite recent progress in understanding its analytic properties.

To illustrate this point, we shall recall some results on the problem of counting cohomological cuspidal automorphic representations of fixed level and growing weight on $GL_2/K$ where $K$ is imaginary quadratic.  This is analogous to Theorem \ref{maincusp}, as these representations are also tempered but not essentially square integrable at infinity.  We let $S_d$ denote the set of cohomological cuspidal automorphic representations of weight $d$ and some fixed unspecified level on $GL_2/K$, where `weight $d$' means having cohomology with respect to the coefficient system $\text{sym}^d \C^2 \otimes \overline{ \text{sym}^d \C^2} \otimes | \text{det}|^{-d}$.  The trivial bound here is $| S_d | \ll d^2$, which is the dimension of the coefficient system.  The best known bound for $| S_d |$ obtained by an analytic study of the trace formula is $| S_d | \ll d^2 / \log d$, due to Finis, Grunewald and Tirao \cite{FGT}.  On the other hand, in \cite{Ma} we used the theory of completed cohomology developed by Calegari and Emerton \cite{CE2,E} to prove that $\dim S_d \ll_\epsilon d^{5/3 + \epsilon}$, and this was later improved to $\ll_\epsilon d^{3/2 +\epsilon}$ by Hu \cite{Hu}.  It is likely that the best bound for $| \cA_\lambda |$ that one could prove using the trace formula is a similar logarithmic improvement over $\lambda^3$.

One has a lower bound for $|\cA_\lambda|$ of $|\cA_\lambda| \gg \lambda$ from symmetric square lifts \cite[Sec 3.4]{AS}, and the computations of \cite{AP} (and those of \cite{FGT} in the analogous case of $SL_2(\C)$) suggest that this is the main contribution so that in fact $|\cA_\lambda| \sim \lambda$.

Theorem \ref{maincoh} will be proved by combining the methods of our previous paper \cite{Ma} with a new bound for the growth of invariants in certain $\F_p$-representations of $SL_3(\Z_p)$ (Proposition \ref{Tinvariants}).  We in fact prove a version of Proposition \ref{Tinvariants} for a general $SL_d(\Z_p)$ (Corollary \ref{Tgencor}), but at present we are unable to deduce new bounds on cohomology from this.  We discuss this further in Section \ref{outline}.

\subsection{Proof of Theorem \ref{maincusp} assuming Theorem \ref{maincoh}}

Before proving Theorem \ref{maincoh}, we show how it implies Theorem \ref{maincusp} using the extension of Matsushima's formula to noncompact quotients proved in \cite{B,BG}.  Let $Z^+$ be the positive scalar matrices in $GL_3(\R)$, and define $X = GL_3(\Q) \backslash GL_3(\A) / KZ^+$.  We have $X = \bigcup \Gamma_i \backslash SL_3(\R)$, where  $\Gamma_i$ are congruence subgroups of $SL(3,\Z)$.

There is a unique irreducible unitary infinite dimensional representation $\pi_\lambda$ of $SL_3(\R)$ with $H^2(\g,K; \pi_\lambda \otimes V_\lambda) \neq 0$.  This implies that if $\pi \in \cA_\lambda$ then the restriction of $\pi_\infty$ to $SL_3(\R)$ must be isomorphic to $\pi_\lambda$.  Moreover, $\pi_\infty$ is trivial on $Z^+$.  If we let $m(\pi_\lambda, X)$ denote the multiplicity with which $\pi_\lambda$ occurs in $L^2_\text{cusp}(X)$, it follows that $|\cA_\lambda| \le m(\pi_\lambda, X)$.  If $\Gamma \subset SL(3,\Z)$ is a congruence subgroup we also let $m(\pi_\lambda, \Gamma)$ be the multiplicity of $\pi_\lambda$ in $L^2_\text{cusp}(\Gamma \backslash SL(3,\R))$.  We have

\bes
m(\pi_\lambda, X) \le \sum_i m(\pi_\lambda, \Gamma_i),
\ees
so it suffices to prove that $m(\pi_\lambda, \Gamma) \ll_{\Gamma, \epsilon} \lambda^{3 - 4/27 + \epsilon}$ for any $\Gamma$.  The extension of Matsushima's formula to noncompact quotients proved in \cite{B,BG} implies that $m(\pi_\lambda, \Gamma) \le \dim H^2(\Gamma, V_\lambda)$, and Theorem \ref{maincoh} completes the proof.

\subsection{Acknowledgements}

This work was supported by NSF grant DMS-1501230.  Part of this work was carried out while the author was the Neil Chriss and Natasha Herron Chriss Founders' Circle Member at the IAS in 2017-18, and we thank both the IAS and the Chriss family for their support.

\section{Proof of Theorem \ref{maincoh}}
\label{sec2}

In this section we prove Theorem \ref{maincoh}, assuming the bound for the growth of invariants in $\F_p$-representations of $SL_3(\Z_p)$ stated in Proposition \ref{Tinvariants}.  We introduce notation in Section \ref{sec:notation}, give  an outline of the proof in Section \ref{outline}, and carry out the details in the rest of the section.  Proposition \ref{Tinvariants} is proved in Section \ref{invariants}.

\subsection{Notation}
\label{sec:notation}

We let $p > 3$ be a prime which will be fixed throughout Section \ref{sec2}.  We define $G_n = \{ g \in SL(3,\Z_p) : g \equiv 1 (p^n) \}$, and let $G = G_1$.  We let $T$ and $U$ be the diagonal and strictly upper triangular subgroups of $G$, and define $P(n) = TUG_n$.  We define the non-commutative Iwasawa algebras
\[
\Lambda_{\Z_p} = \underset{ \substack{ \longleftarrow \\ k } }{\lim} \: \Z_p[ G / G_k ], \quad \Lambda_{\Q_p} = \Lambda_{\Z_p} \otimes_{\Z_p} \Q_p,  \quad \text{and} \quad \Lambda = \underset{ \substack{ \longleftarrow \\ k } }{\lim} \: \F_p[ G / G_k ],
\]
where the projections are given by the trace maps $\Z_p[ G / G_{k'} ] \rightarrow \Z_p[ G / G_k ]$ for $k' \ge k$.  Suppose $L$ is a representation of $G$ over $\F_p$, and let $L^*$ be the dual representation equipped with the weak topology.  If $L$ is smooth, then the action of $G$ on $L^*$ extends uniquely to an action of $\Lambda$ such that for any $\ell \in L^*$ the orbit map $\Lambda \to L^*$, $x \mapsto x \ell$, is continuous.

\subsection{Outline of proof}
\label{outline}

We now sketch the proof of Theorem \ref{maincoh}.  The first ingredient is the following bound for the growth of invariants in representations of $G$.

\begin{prop}
\label{Tinvariants}

Let $L$ be a smooth admissible representation of $G$ over $\F_p$ such that $L^*$ is a finitely generated torsion $\Lambda$ module.  Then we have

\bes
\dim L^{T G_n} \ll_L (100 p^{-4/9})^n | G : T G_n|
\ees
for all $n$.

\end{prop}

Note that Proposition \ref{Tinvariants} represents a power saving over the trivial bound of $\dim L^{T G_n} \ll_L | G : T G_n|$.  It will be proved in Section \ref{invariants}.  The rest of the proof uses the same method as \cite{Ma}, which we summarize below.  Throughout, we shall consider $V_\lambda$ as a representation of $GL_3(\Q)$ with $\Q_p$ coefficients.  We use $h^i$ to denote the dimension of $H^i$, computed with continuous cochains in the case of the group $G$.

\setlength{\leftmargini}{1em}
\begin{enumerate}[(i)]

\item
\label{step1}
The theory of completed cohomology gives $h^2(\Gamma, V_\lambda) \le h^0( G, \widetilde{H}^2_{\Q_p} \otimes_{\Q_p} V_\lambda)$, where $\widetilde{H}^2_{\Q_p}$ is the $p$-adically completed $H^2$ defined in Section \ref{sec:compcoh}.  It is a $\Q_p$-Banach space with a continuous action of $G$.

\item
\label{step2}
If $n$ is the smallest integer such that $p^{n-1} > 3\lambda$, by Lemma \ref{lattice} we may choose a $G$-stable lattice $\cV_\lambda \subset V_\lambda$ such that $\cV_\lambda / p$ embeds as a subrepresentation of $\F_p[G / P(n)]$.

\item
\label{step3}
There is a $G$-stable lattice $L \subset \widetilde{H}^2_{\Q_p}$, and reduction mod $p$ gives $h^2(\Gamma, V_\lambda) \le h^0(G, \overline{L} \otimes_{\F_p} \cV_\lambda / p)$, where $\overline{L} := L / pL$ is a smooth admissible representation of $G$ over $\F_p$ such that $\overline{L}^*$ is a finitely generated $\Lambda$-module.

\item
\label{step4}
The embedding $\cV_\lambda / p \subset \F_p[G / P(n)]$ and Shapiro's lemma give $h^0(G, \overline{L} \otimes_{\F_p} \cV_\lambda / p) \le \dim \overline{L}^{P(n)}$.

\item
\label{step5}
The trace formula, and the fact that $SL(3,\R)$ does not have discrete series, imply that $\overline{L}^*$ is a torsion $\Lambda$-module.

\item
\label{step6}
Conjugating the subgroup $P(n)$ by an element of $SL_3(\Q_p)$ and applying Proposition \ref{Tinvariants} gives $\dim \overline{L}^{P(n)} \ll p^{(3-4/27)n} \sim \lambda^{3-4/27}$, which completes the proof.

\end{enumerate}

The fact that we cannot presently obtain Theorem \ref{maincusp} for extensions of $\Q$, or higher $GL_d$, is due to the way we obtain $h^2(\Gamma, V_\lambda) \le h^0( G, \widetilde{H}^2_{\Q_p} \otimes_{\Q_p} V_\lambda)$ in step (\ref{step1}).  We do this starting from the bound
\bes
h^2(\Gamma, V_\lambda) \le \sum_{i+j = 2} h^i( G, \widetilde{H}^j_{\Q_p} \otimes_{\Q_p} V_\lambda),
\ees
where $\widetilde{H}^j_{\Q_p}$ is again the $p$-adically completed $H^j$.  We then apply the congruence subgroup property for $SL(3,\Z)$ to show that the contributions from $\widetilde{H}^0_{\Q_p}$ and $\widetilde{H}^1_{\Q_p}$ are trivial, so we are left with $h^0( G, \widetilde{H}^2_{\Q_p} \otimes_{\Q_p} V_\lambda)$.
The fact that only an $h^0$ remains is essential to our argument, as it lets us use the bound $h^0(G, \overline{L} \otimes_{\F_p} \cV_\lambda / p) \le h^0(G, \overline{L} \otimes_{\F_p} \F_p[G / P(n)] )$ in step (\ref{step4}) (which isn't necessarily true for higher $h^i$).

To generalize our argument, we would need to bound the growth of $h^a(\Gamma, V)$ in the lowest degree $a$ that cusp forms contribute (where $V$ is a varying local system).  For a group $GL_d / F$ with $d \ge 4$, or $d = 3$ and $F \neq \Q$, we have $a \ge 3$.  Therefore, to control the right hand side of the inequality
\bes
h^a(\Gamma, V) \le \sum_{i+j = a} h^i( G, \widetilde{H}^j_{\Q_p} \otimes_{\Q_p} V)
\ees
we would need to control a term $h^i( G, \widetilde{H}^j_{\Q_p} \otimes_{\Q_p} V)$ with $i \ge 1$ and $j \ge 2$, and we currently do not know how to do this.  In \cite{Ma}, we overcame this obstacle by showing that $\cV_\lambda/p$ had an efficient filtration by modules isomorphic to $\F_p[G / P(n) ]$ for varying $n$ (or rather, the analogous statement in the $GL_2$ case).  In the case of $GL_3$, it would suffice to solve the following:

\begin{problem}

There is $\delta > 0$ such that for any $\lambda$, $\cV_\lambda / p$ has a filtration by modules $\F_p[G / P(n_i)]$ such that
\[
\sum_i | G : P(n_i) |^{1- 4/81} \ll \lambda^{3 - \delta}.
\]

\end{problem}

Note that the exponent $1 - 4/81$ comes from the bound $\dim \overline{L}^{P(n)} \ll | G : P(n) |^{1 - 4/81}$ appearing in step (\ref{step6}) above.

\subsection{Choosing a lattice in $V_\lambda$}

We now find a lattice $\cV_\lambda \subset V_\lambda$ with the properties described above.

\begin{lemma}
\label{lattice}

If $n \ge 1$ satisfies $p^{n-1} > 3\lambda$, there is a $G$-stable lattice $\cV_\lambda \subset V_\lambda$ such that $\cV_\lambda / p\cV_\lambda$ is isomorphic to a submodule of $\F_p[G / P(n)]$.

\end{lemma}

\begin{proof}

Let $V_\lambda^*$ be the dual of $V_\lambda$, and $\langle \cdot, \cdot \rangle$ the pairing between them.  Let $w_\lambda^* \in V_\lambda^*$ be a nonzero vector of highest weight. As a representation of $SL_3(\Q_p)$, $V_\lambda$ is isomorphic to the space of functions on $SL(3,\Q_p)$ of the form

\bes
f(g) = \langle \pi(g^{-1}) v, w_\lambda^* \rangle, \quad v \in V_\lambda,
\ees
where $SL(3,\Q_p)$ acts by $[\pi(h)f](g) = f(h^{-1} g)$.  We define $\cV_\lambda$ to be the $\Z_p$-module of functions whose values on $G$ lie in $\Z_p$, which is clearly a $G$-stable lattice.  This implies that $\cV_\lambda / p\cV_\lambda$ may be identified with the submodule of $C(G, \F_p)$ obtained by reducing functions in $\cV_\lambda$ modulo $p$, and we must show that these reductions are right-invariant under $P(n)$.

Let $w_{-\lambda} \in V_\lambda$ be the vector of lowest weight with $\langle w_{-\lambda}, w_\lambda^* \rangle = 1$, and define $f_\lambda(g) = \langle \pi(g^{-1}) w_{-\lambda}, w_\lambda^* \rangle$.  We have $f_\lambda(e) = 1$, and the invariance properties of $w_{-\lambda}$ and $w_\lambda^*$ imply that for $u_-$ in the lower triangular unipotent and $b$ in the standard Borel we have $f(u_- b) = \chi(b)$, where $\chi$ is the highest weight character of $V_\lambda^*$.  By the Bruhat decomposition, these facts specify $f_\lambda$ uniquely.  Moreover, the function

\bes
h_\lambda : \left( \begin{array}{ccc} x_{11} & x_{12} & x_{13} \\ x_{21} & x_{22} & x_{23} \\ x_{31} & x_{32} & x_{33} \end{array} \right) \to x_{11}^\lambda \det \left( \begin{array}{cc} x_{11} & x_{12} \\ x_{21} & x_{22} \end{array} \right)^{\lambda}
\ees
also has these properties, so that $f_\lambda = h_\lambda$.

Let $P_\lambda$ be the space of polynomials on $M_3(\Q_p)$ spanned by the left translates of $h_\lambda$ under $SL_3(\Q_p)$, and let $\cP_\lambda \subset P_\lambda$ be the lattice of polynomials that are integral on $G$.  It follows that $V_\lambda$ and $\cV_\lambda$ are the restrictions of $P_\lambda$ and $\cP_\lambda$ to $SL_3(\Q_p)$ respectively.  Moreover, because all elements of $P_\lambda$ transform by the same character under the action of scalar matrices, elements of $\cP_\lambda$ are integral on $1 + p M_3(\Z_p)$.  We also see that all polynomials in $P_\lambda$ have degree at most $3\lambda$.  A theorem of Lucas \cite{Lu} states that if $h(x)$ is an integer valued polynomial on $\Z_p$ of degree $d$, and $d \le p^t$, then the reduction of $h$ modulo $p$ is constant on cosets of $p^t \Z_p$.  This implies that, on $1 + p M_3(\Z_p)$, the functions in $\cP_\lambda / p\cP_\lambda$ are constant on cosets of $p^n M_3(\Z_p)$, and hence that functions in $\cV_\lambda / p \cV_\lambda$ are invariant under $G_n$.  The invariance under $TU$ follows from our choice of $w_\lambda^*$ as a highest weight vector, which completes the proof.

\end{proof}

\subsection{$p$-adic Banach space representations}

We now recall from \cite{ST} some facts about representations of $G$ (or any other compact $p$-adic Lie group) on a $p$-adic Banach space.

We recall that a topological $\Z_p$-module is called linear-topological if the zero element has a fundamental system of open neighbourhoods consisting of $\Z_p$-submodules.  We let $\text{Mod}_\text{top}(\Z_p)$ be the category of all Hausdorff linear-topological $\Z_p$-modules with morphisms being all continuous $\Z_p$-linear maps.  We let $\text{Mod}^\text{tf}_\text{c}(\Z_p)$ be the full subcategory in $\text{Mod}_\text{top}(\Z_p)$ consisting of torsion free and compact linear-topological $\Z_p$-modules.  We recall from \cite[Rmk 1.1]{ST} that if $M$ is any object in $\text{Mod}^\text{tf}_\text{c}(\Z_p)$, we have
\be
\label{freemod}
 M \simeq \prod_{i \in I} \Z_p
 \ee
 for some set $I$.  We let $\text{Ban}(\Q_p)$ be the category of all $\Q_p$-Banach spaces $(E, \| \; \|)$, with morphisms being continuous linear maps.  We let $\text{Ban}(\Q_p)^{\le 1}$ be the subcategory of spaces such that $\| E \| \subset |\Q_p|$, with the morphisms being norm-decreasing linear maps.

In \cite[Section 1]{ST}, Schneider and Teitelbaum define two contravariant functors between the categories $\text{Ban}(\Q_p)^{\le 1}$ and $\text{Mod}^\text{tf}_\text{c}(\Z_p)$.  For an object $M$ in $\text{Mod}^\text{tf}_\text{c}(\Z_p)$ they define the $\Q_p$-Banach space
\[
M^d = \text{Hom}_{\Z_p}^\text{cts}(M, \Q_p) \quad \text{with norm} \quad \| \ell \| = \underset{m \in M}{\max} \; | \ell(m) |,
\]
which defines a contravariant functor $\text{Mod}^\text{tf}_\text{c}(\Z_p) \to \text{Ban}(\Q_p)^{\le 1}$.  For a Banach space $(E, \| \; \|)$ with unit ball $L$, they also define
\[
E^d = \text{Hom}_{\Z_p}( L, \Z_p) \quad \text{with the topology of pointwise convergence},
\]
which gives a contravariant functor $\text{Ban}(\Q_p)^{\le 1} \to \text{Mod}^\text{tf}_\text{c}(\Z_p)$.  They prove in \cite[Thm 1.2]{ST} that the functors $M \mapsto M^d$ and $E \mapsto E^d$ define an antiequivalence of categories.  In particular, if $E$ is an object in $\text{Ban}(\Q_p)^{\le 1}$, $L$ is the unit ball in $E$, and $M = E^d$, then we have $E = M^d$ and $L = \text{Hom}_{\Z_p}^\text{cts}(M, \Z_p)$.  

We now consider a Banach space $E$ with a continuous representation of $G$.  We say that this representation is unitary if it preserves the norm.  The representation induces an action of $G$ on the dual $E'$, which may be completed to an action of $\Lambda_{\Q_p}$.  We say that $E$ is admissible as a representation of $G$ if $E'$ is a finitely generated $\Lambda_{\Q_p}$-module.  In \cite[Lemma 3.4]{ST} and the subsequent discussion, the authors show that admissibility implies that for any $G$-stable lattice $L \subset E$, the representation of $G$ on $L / pL$ is smooth and admissible (in the ordinary sense that $(L/pL)^H$ is finite dimensional for any open $H < G$).   They also prove that if an object $E$ in $\text{Ban}(\Q_p)^{\le 1}$ carries an admissible unitary representation of $G$, then $E^d$ is a finitely generated $\Lambda_{\Z_p}$-module.

\subsection{Completed cohomology}
\label{sec:compcoh}

We now complete the proof of Theorem \ref{maincoh}.  There is an injection $\phi : \Gamma \to SL(3,\Z_p)$ such that $\overline{\phi(\Gamma)}$ is open.  By choosing $p$ sufficiently large and passing to a subgroup of $\Gamma$, we may assume that $\overline{\phi(\Gamma)} = G$.  Our assumption that $p > 3$ then implies that $\Gamma$ is torsion free.  For $k \ge 1$, define $\Gamma_k = \Gamma \cap G_k$.  Let $n$ be the smallest integer with $p^{n-1} > 3\lambda$ and $3 | n$, and let $\cV_\lambda \subset V_\lambda$ be obtained by applying Lemma \ref{lattice} to this $n$.  Following Calegari and Emerton, we define

\bes
\widetilde{H}^i(\cV_\lambda) = \underset{ \substack{ \longleftarrow \\ s } }{\lim} \underset{ \substack{ \longrightarrow \\ k } }{\lim} H^i( \Gamma_k, \cV_\lambda / p^s), \quad \widetilde{H}^i(\cV_\lambda)_{\Q_p} = \widetilde{H}^i(\cV_\lambda) \otimes _{\Z_p} \Q_p.
\ees
Because $\cV_\lambda$ is continuous as a representation of $G$, for each fixed $s$, $\cV_\lambda / p^s$ is eventually trivial on $\Gamma_k$.  If we define

\bes
\widetilde{H}^i = \underset{ \substack{ \longleftarrow \\ s } }{\lim} \underset{ \substack{ \longrightarrow \\ k } }{\lim} H^i( \Gamma_k, \Z_p / p^s), \quad \widetilde{H}^i_{\Q_p} = \widetilde{H}^i \otimes_{\Z_p} \Q_p,
\ees
we therefore have $\widetilde{H}^i(\cV_\lambda) = \widetilde{H}^i \otimes_{\Z_p} \cV_\lambda$ and $\widetilde{H}^i(\cV_\lambda)_{\Q_p} = \widetilde{H}^i_{\Q_p} \otimes_{\Q_p} V_\lambda$.  We recall the following facts about $\widetilde{H}^i_{\Q_p}$.

\setlength{\leftmargini}{2em}
\begin{enumerate}

\item $\widetilde{H}^i$ is a $p$-adically separated and complete $\Z_p$-module, with bounded $p$-torsion exponent \cite[Thm 1.1]{CE2} and \cite[Lemma 2.1.4]{E}.  This implies that the torsion free quotient $\widetilde{H}^i_\text{tf}$ is also separated and complete, so that $\widetilde{H}^i_{\Q_p} = \widetilde{H}^i_\text{tf} \otimes \Q_p$ is naturally a $\Q_p$-Banach space in which $\widetilde{H}^i_\text{tf}$ is the unit ball.

\item The natural action of $G$ on $\widetilde{H}^i_\text{tf}$ induces an admissible unitary representation of $G$ on $\widetilde{H}^i_{\Q_p}$ which extends to the group $SL(3,\Q_p)$ \cite[Thm 2.1.5 and 2.2.11]{E}.

\item Because $SL(3,\R)$ does not admit discrete series, the dual space of $\widetilde{H}^i_{\Q_p}$ is a torsion $\Lambda_{\Q_p}$-module for all $i$ \cite{CE1}.

\item
\label{spectral}
There is a spectral sequence $E^{i,j}_2 = H^i_\text{cts}( G, \widetilde{H}^j_{\Q_p} \otimes_{\Q_p} V_\lambda) \implies H^{i+j}(\Gamma, V_\lambda)$ \cite[Prop 2.1.11]{E}.

\end{enumerate}

\noindent
The spectral sequence above implies that

\bes
h^2(\Gamma, V_\lambda) \le \sum_{i+j = 2} h^i( G, \widetilde{H}^j_{\Q_p} \otimes_{\Q_p} V_\lambda),
\ees
where $h_i$ denotes the dimension of $H^i_\text{cts}$.  Moreover, the terms in this sum with $j = 0,1$ vanish.  Indeed, one has $\widetilde{H}^0 = \Z_p$, and so $h^2(G, \widetilde{H}^0_{\Q_p} \otimes_{\Q_p} V_\lambda) = h^2(G, V_\lambda) = 0$.  Also, one has $\widetilde{H}^1 = 0$ by the congruence subgroup property for $SL(3,\Z)$.  We therefore have

\bes
h^2(\Gamma, V_\lambda) \le h^0( G, \widetilde{H}^2_{\Q_p} \otimes_{\Q_p} V_\lambda).
\ees
We let $E = \widetilde{H}^2_{\Q_p}$, and let $L = \widetilde{H}^2_\text{tf}$ be the unit ball in $E$.  We define $\overline{L} = L / pL$.  We then have

\bes
h^0( G, E \otimes_{\Q_p} V_\lambda) \le h^0( G, L \otimes_{\Z_p} \cV_\lambda ) \le h^0( G, \overline{L} \otimes_{\F_p} (\cV_\lambda / p) ),
\ees
where the middle term denotes the rank of the free finitely generated $\Z_p$-module $(L \otimes_{\Z_p} \cV_\lambda)^G$.  The inclusion $\cV_\lambda  / p \cV_\lambda \subset \F_p[G / P(n)]$ from Lemma \ref{lattice} gives 
\[
h^0( G, \overline{L} \otimes_{\F_p} (\cV_\lambda/p) ) \le h^0( G, \overline{L} \otimes_{\F_p} \F_p[G / P(n)] ),
\]
and by Shapiro's Lemma this is equal to $\dim \overline{L}^{P(n)}$.  We now use the $SL(3,\Q_p)$ action to conjugate $P(n)$ so that it is closer to $T G_n$, which lets us apply Proposition \ref{Tinvariants}.  If we define $x = \text{diag}(p^{n/3}, 1, p^{-n/3})$, then we have $x P(n) x^{-1} \subset T G_{ n/3 }$.  We apply Lemma \ref{subgroup} to these groups, which gives

\bes
\dim \overline{L}^{P(n)} = \dim \overline{L}^{x P(n) x^{-1}}  \le |T G_{ n/3 } : x P(n) x^{-1}| \dim \overline{L}^{T G_{ n/3 }}.
\ees
We let $M = E^d$, so that $L \simeq \Hom_\text{cts}(M, \Z_p)$.  The isomorphism $M \simeq \prod_{i \in I} \Z_p$ from (\ref{freemod}) implies that
\[
L \simeq c_0(I, \Z_p) := \{ f : I \to \Z_p \; | \; \text{for all } c > 0,  | f(i) | > c \text{ for only finitely many } i \},
\]
and that the vector spaces $\overline{L}$ and $\overline{M} := M / pM$ satisfy $\overline{M} = \overline{L}^*$ (where the quotient topology on $\overline{M}$ is the same as the pointwise topology).  We know that $M$ is a finitely generated and torsion $\Lambda_{\Z_p}$-module, and because it has no $\Z_p$-torsion this implies that $\overline{M}$ is a finitely generated and torsion $\Lambda$-module.  Proposition \ref{Tinvariants} then gives

\begin{align*}
|T G_{ n/3 } : x P(n) x^{-1}| \dim \overline{L}^{T G_{ n/3 }} & \ll (100 p^{-4/9})^{n/3} |G : x P(n) x^{-1}| \\
& = (100 p^{-4/9})^{n/3} p^{3n}.
\end{align*}
By our choice of $n$, $(100 p^{-4/9})^{n/3} p^{3n} \ll \lambda^{3 - 4/27} 10^n$, which completes the proof after choosing $p$ sufficiently large.

\section{Invariants of $\Lambda$ modules}
\label{invariants}

This section contains the proof of Proposition \ref{Tinvariants}.  We in fact prove a general version for any $SL_d(\Z_p)$, stated as Corollary \ref{Tgencor} below.  For $d \ge 1$, let $G_d(n) = \{ g \in SL_d(\Z_p) : g \equiv 1 (p^n) \}$ and $G_d = G_d(1)$.  As $d$ will be fixed for most of the proof, we shall usually omit it and simply write $G$ and $G(n)$.  We define $T$ and $\Lambda$ in the analogous way to Section \ref{sec:notation}.  We shall deduce Corollary \ref{Tgencor} from a theorem of M. Harris \cite{H}, and the following theorem, whose proof uses only elementary representation theory.

\begin{theorem}
\label{Tgen}

Assume that $p > d$.  Let $L$ be a representation of $G$ over $\F_p$ such that $\dim L^{G(n)} \ll_L p^{-n} | G : G(n) |$ for all $n$.  We then have

\bes
\dim L^{T G(n)} \ll_L 10^{(d-1) n} p^{-(2/3)^{d-1} n} |G : T G(n) |.
\ees

\end{theorem}

The main result of \cite{H} implies that if $L$ is a smooth admissible representation of $G$ over $\F_p$ such that $L^*$ is a finitely generated torsion $\Lambda$ module, then $L$ satisfies the hypothesis of Theorem \ref{Tgen}.  We therefore have:

\begin{cor}
\label{Tgencor}

Assume that $p > d$.  Let $L$ be a smooth admissible representation of $G$ over $\F_p$ such that $L^*$ is a finitely generated torsion $\Lambda$ module.  Then

\bes
\dim L^{T G(n)} \ll_L 10^{(d-1) n} p^{-(2/3)^{d-1} n} |G : T G(n) |.
\ees

\end{cor}

Note that if $L^*$ is a finitely generated $\Lambda$ module, then one has the trivial bounds $\dim L^{G(n)} \ll | G : G(n) |$ and $\dim L^{T G(n)} \ll | G : T G(n) |$, and the results of \cite{H} and Corollary \ref{Tgencor} respectively represent power savings over these bounds under the assumption that $L^*$ is torsion.  To prove Theorem \ref{Tgen}, we shall show that one may stretch the subgroup $G(n)$ into $T G(n)$ while maintaining control of the invariants, using the basic method of \cite[Prop 7]{Ma}.

\subsection{A lemma on passage to subgroups}

The following lemma will let us pass bounds for invariants between subgroups of $G$.

\begin{lemma}
\label{subgroup}

Let $V$ be a representation of $G$ over $\F_p$, and let $G \ge H_1 \ge H_2$ be open subgroups of $G$.  We have $\dim V^{H_2} \le | H_1 : H_2 | \dim V^{H_1}$.

\end{lemma}

\begin{proof}

By Lemma \ref{pindex}, it suffices to find a chain of normal subgroups $H_1 = J_1 \rhd J_2 \rhd \ldots \rhd J_i = H_2$.  We claim that the groups $J_k = (H_1 \cap G(k)) H_2$ (which stabilize at $H_2$ for $k$ large) suffice.  First, one observes that these are in fact groups, as $H_2$ normalizes $H_1$ and $G(k)$.  Next, we wish to show that $J_k$ is normal in $J_{k-1}$.  To do this, it suffices to check that $H_2$ and $H_1 \cap G(k-1)$ each normalize $J_k$.  This is clear for $H_2$, as it normalizes $H_1$, $H_2$, and $G(k)$.  Moreover, $H_1 \cap G(k-1)$ normalizes $H_1 \cap G(k)$, and so it suffices to show that for $g \in H_1 \cap G(k-1)$ and $h \in H_2$ we have $g h g^{-1} \in J_k$.  We have $[g, h] \in [ G(k-1), G ] \subset G(k)$ and $[g, h] \in [ H_1, H_2] \subset H_1$, hence $g h g^{-1} \in (H_1 \cap G(k)) h \subset J_k$ as required.


\end{proof}

\begin{lemma}
\label{pindex}

Let $J_1 \rhd J_2$ be two groups, with $J_1 / J_2$ of order $p$.  Let $V$ be a representation of $J_1$ over $\F_p$.  Then $\dim V^{J_2} \le p \dim V^{J_1}$.

\end{lemma}

\begin{proof}

The space $V^{J_2}$ carries a representation of $J_1 / J_2$.  If we let $j \in J_1 / J_2$ be nontrivial, then on $V^{J_2}$ we have $\ker(1 - j) = V^{J_1}$ and $(1 - j)^p = 0$.  The lemma follows.

\end{proof}

\subsection{Proof of Theorem \ref{Tgen}.}

Let $S \subset T$ be the torus
\[
S = \left\{ \left( \begin{array}{cccc} x &&& \\ & \ddots && \\ && x & \\ &&& x^{1-d} \end{array}\right) : x \in 1 + p \Z_p \right\}.
\]
We first prove a version of Theorem \ref{Tgen} with $T$ replaced by $S$.

\begin{prop}
\label{Sinvariants}

Assume that $p > d$.  Let $V$ be a representation of $G$ over $\F_p$, and suppose that there exists $C, N > 0$ and $R < p^{d^2-1}$ such that $\dim V^{G(n)} \le C R^n$ for all $n \le N$.  Let $\rho = R^{1/3} p^{(4-d^2)/3}$.  Then
\[
\dim V^{S G(n)} \le C' (10 \rho^{-1})^n R^n = C' 10^n R^{2n/3} p^{(d^2-4)n/3}
\]
for all $n \le N$, where $C' = C \max \{ 1, (10 \rho^{-1})^{-2} \}$.

\end{prop}

We could prove Proposition \ref{Sinvariants} with any one dimensional torus, but the reason we have chosen $S$ is that it commutes with the copy of $G_{d-1}$ in the upper left hand block of $G$.  It follows that $V^{S G(n)}$ is a $G_{d-1}$ module, which lets us apply Proposition \ref{Sinvariants} inductively to bound $\dim V^{T G(n)}$ in Proposition \ref{Tind} below.

Before proving Proposition \ref{Sinvariants}, we shall illustrate the basic idea using a toy example.  We assume that $d = 3$ until further notice.  For $n \ge 2$, let $G^+(n) = (S \cap G(n-1)) G(n)$, so that $G^+(n)$ represents a line in the $S$ direction in $G(n-1) / G(n) \simeq \F_p^8$.

\begin{lemma}

If $p > 3$ and $V$ is a representation of $SL_3(\Z_p)$ over $\F_p$, we have $\dim V^{G^+(n)} \le \max \{ \dim V^{G(n)} -1, p^3 \dim V^{G(n-1)} \}$ for any $n \ge 2$.

\end{lemma}

\begin{proof}

If $\dim V^{G^+(n)} > \dim V^{G(n)} -1$, then $V^{G^+(n)} = V^{G(n)}$.  It follows that any vector in $V^{G(n)}$ must also be invariant under any conjugate of $G^+(n)$.  Because $p > 3$, the $SL_3(\Z_p)$ conjugates of $G^+(n)$ span a 5-dimensional subspace of $G(n-1) / G(n)$, which we call $X$.  Lemma \ref{subgroup} gives
\[
\dim V^{G^+(n)} = \dim V^X \le | G(n-1) : X | \dim V^{G(n-1)} = p^3 \dim V^{G(n-1)},
\]
which completes the proof.

\end{proof}

Roughly speaking, we will combine this idea with inclusion-exclusion counting, applied to the subspaces $V^{g G^+(n) g^{-1}}$ with $g \in SL_3(\Z_p)$.  This gives us a good bound for $\dim V^{G^+(n)}$, which forms the first step of an induction argument that we use to pass from $G(n)$ to $S G(n)$ one congruence step at a time.  It might be possible to improve the bound we get by taking more conjugates of $S$ than we do here.

\begin{proof}[Proof of Proposition \ref{Sinvariants} ]

We now let $d$ be arbitrary again.  If $\rho < 10$ then the bound we wish to prove is weaker than the trivial bound $\dim V^{S G(n)} \le C R^n$.  We may therefore assume that $\rho \ge 10$, in which case the bound we must prove is $\dim V^{S G(n)} \le C (10 \rho^{-1})^{n-2} R^n$.  For any $n-1 \ge k \ge 0$, we define $S(n,k) = (S \cap G(n-k)) G(n)$.  One may think of $S(n,k)$ as the subgroup of $G$ obtained by stretching $G(n)$ by $k$ steps in the $S$ direction.  We shall prove by induction that

\be
\label{indhyp}
\dim V^{S(n,k)} \le C (10 \rho^{-1})^{k-1} R^{n}
\ee
for all $1 \le n \le N$ and $0 \le k \le n-1$.  As $S(n,n-1) = S G(n)$, this gives the proposition.  Note that (\ref{indhyp}) follows from the conditions of the proposition when $k = 0,1$.

Fix $(n,k)$, and suppose that (\ref{indhyp}) holds for all $(n',k')$ less than $(n,k)$ in the lexicographic ordering.  We may assume that $k \ge 2$, and hence that $n \ge 3$.  As in \cite[Prop 7]{Ma}, we shall deduce (\ref{indhyp}) for $(n,k)$ from the cases $(n-1, k-1)$ and $(n,k-1)$, by applying inclusion-exclusion counting to the invariants under certain subgroups lying between $S(n-1, k-1)$ and $S(n, k-1)$.

It may be seen that $S(n,k-1)$ is normal in $S(n-1, k-1)$, and that the quotient $X = S(n-1, k-1) / S(n, k-1)$ is Abelian and isomorphic to the vector space $\F_p^{d^2-1}$.  The image of $S(n,k)$ in $X$ is a line, which we denote by $\ell$.  We define

\bes
N = \left( \begin{array}{cccc} 1 &&& \\ & \ddots && \\ && 1 & p^{n-1} \\ &&& 1 \end{array} \right), \quad \overline{N} = \left( \begin{array}{cccc} 1 &&& \\ & \ddots && \\ && 1 & \\ && p^{n-1} & 1 \end{array} \right),
\ees  
and define $W \subset X$ to be the subspace spanned by $\ell$, $N$, and $\overline{N}$.  Define $U \subset W$ to be the subspace spanned by $N$ and $\overline{N}$.  If $Y \subset X$ is any subspace (which we may identify with a subgroup of $G$), we let $V^Y$ be the vectors in $V$ fixed by $Y$.  The argument on \cite[p. 1638]{Ma} gives

\be
\label{inex}
\left( \frac{2m}{m-1} - 1 \right) \dim V^\ell \le \frac{2}{m-1} \dim V^0 + \frac{m(m-1)}{2} \dim V^W,
\ee
where $m = \lfloor \rho \rfloor$, and we briefly recall how this works.  First, the following lemma implies that $\dim V^{\ell'} = \dim V^{\ell}$ for any line $\ell' \subset W$ not contained in $U$.

\begin{lemma}

If $\ell' \subset W$ is a line not contained in $U$, then there is $g \in G$ whose action by conjugation descends to $X$, and such that $g \ell' g^{-1} = \ell$.

\end{lemma}

\begin{proof}

This follows in the same way as \cite[Lemma 8]{Ma}.  If we define

\bes
N' = \left( \begin{array}{cccc} 1 &&& \\ & \ddots && \\ && 1 & p^{k-1} \\ &&& 1 \end{array} \right), \quad \overline{N}' = \left( \begin{array}{cccc} 1 &&& \\ & \ddots && \\ && 1 & \\ && p^{k-1} & 1 \end{array} \right),
\ees
it may be checked that $N'$ and $\overline{N}'$ normalize $S(n-1, k-1)$ and $S(n, k-1)$, and that conjugation by $N'$ or $\overline{N}'$ acts on $W$ by shearings that fix $U$ pointwise and translate in the directions of $N$ and $\overline{N}$ respectively.

\end{proof}

Next, if $P \subset W$ is a plane different from $U$, and $\ell_1, \ldots,\ell_j \subset P$ are distinct lines that do not lie in $U$, then \cite[Lemma 9]{Ma} gives
\be
\label{lines}
\dim V^{\ell_1} + \dim \sum_{i=2}^j V^{\ell_i} \le \dim \sum_{i=1}^j V^{\ell_i} + (j-1) \dim V^P.
\ee
The assumption $R < p^{d^2-1}$ implies that $m \le \rho < p$, so that we may choose $m$ lines $\ell_1, \ldots, \ell_m$ satisfying these conditions.  We may apply (\ref{lines}) successively to the collections $\{ \ell_1, \ldots, \ell_m \}, \{ \ell_2, \ldots, \ell_m \}, \ldots, \{ \ell_{m-1}, \ell_m \}$ to obtain
\begin{align*}
\dim V^{\ell_1} + \dim \sum_{i=2}^m V^{\ell_i} & \le \dim \sum_{i=1}^m V^{\ell_i} + (m-1) \dim V^P \\
\dim V^{\ell_2} + \dim \sum_{i=3}^m V^{\ell_i} & \le \dim \sum_{i=2}^m V^{\ell_i} + (m-2) \dim V^P \\
& \vdots \\
\dim V^{\ell_{m-1}} + \dim V^{\ell_m} & \le \dim( V^{\ell_{m-1}} + V^{\ell_m}) + \dim V^P.
\end{align*}
Adding these and simplifying gives
\be
\label{linesum}
\sum_{i=1}^m \dim V^{\ell_i} \le \dim \sum_{i=1}^m V^{\ell_i} + \frac{m(m-1)}{2} \dim V^P \le \dim V^0 + \frac{m(m-1)}{2} \dim V^P.
\ee
When combined with $\dim V^{\ell_i} = \dim V^\ell$, this becomes
\be
\label{lineineq}
m \dim V^{\ell} \le \dim V^0 + \frac{m(m-1)}{2} \dim V^P.
\ee
If $P_1, \ldots, P_m \subset W$ are planes containing $\ell$, we may apply the argument from \cite[Lemma 9]{Ma} to the lines $P_1 / \ell, \ldots, P_m / \ell$ in $W / \ell$ to obtain the analog of (\ref{lines}), and hence of (\ref{linesum}), which is
\[
\sum_{i=1}^m \dim V^{P_i} \le \dim V^{\ell} + \frac{m(m-1)}{2} \dim V^W.
\]
Bounding each $\dim V^{P_i}$ from below using (\ref{lineineq}) and rearranging gives (\ref{inex}).

Our inductive hypothesis (\ref{indhyp}) for $(n-1,k-1)$ gives $\dim V^X \le C (10 \rho^{-1})^{k-2} R^{n-1}$, and combining this with Lemma \ref{subgroup} we have
\[
\dim V^W \le p^{d^2-4} \dim V^X \le C (10 \rho^{-1})^{k-2} R^{n-1} p^{d^2-4}.
\]
The inductive hypothesis for $(n,k-1)$ gives $\dim V^0 \le C (10 \rho^{-1})^{k-2} R^{n}$, and substituting these into (\ref{inex}) gives

\bes
\left( \frac{2m}{m-1} - 1 \right) \dim V^\ell \le C (10 \rho^{-1})^{k-2} R^n \left( \frac{2}{m-1}  + \frac{m(m-1)}{2} R^{-1} p^{d^2-4} \right).
\ees
By our choice of $m$, we have

\begin{multline*}
\left( \frac{2m}{m-1} - 1 \right)^{-1} \left( \frac{2}{m-1}  + \frac{m(m-1)}{2} R^{-1} p^{d^2-4} \right) \\
= \left( \frac{2m}{m-1} - 1 \right)^{-1} \left( \frac{2}{m-1}  + \frac{m(m-1)}{2} \rho^{-3} \right) \le 10 \rho^{-1}.
\end{multline*}
This completes the inductive step, and hence the proof.

\end{proof}

\begin{prop}
\label{Tind}

Assume that $p > d$.  Let $V$ be a representation of $G$ over $\F_p$, and suppose that there exists $C, N > 0$ and $R < p^{d^2-1}$ such that $\dim V^{G(n)} \le C R^n$ for all $n \le N$.  Then
\[
\dim V^{T G(n)} \le C \kappa 10^{(d-1)n} R^{(2/3)^{d-1} n} p^{\sigma(d) n}
\]
for all $n \le N$, where $\sigma(d) = d(d-1) - (2/3)^{d-1}(d^2-1)$, and $\kappa > 0$ depends only on $R$, $p$, and $d$.

\end{prop}

\begin{proof}

We proceed by induction on $d$, using Proposition \ref{Sinvariants}.  We let $\kappa > 0$ denote a constant depending only on $R$, $p$, and $d$ that may vary from line to line.  The base case of $d = 2$ is exactly the statement of Proposition \ref{Sinvariants}.  We next prove it for a given $d \ge 3$, assuming it holds for $d-1$.  Let $V$, $N$, $R$, and $C$ be as in the statement of the proposition.  We think of the groups $G_{d-1}(n)$ as embedded in $G_d$ in the upper left block.  For any $k \le n \le N$, we may apply Lemma \ref{subgroup} to the groups $S G_d(k) \geqslant S G_d(n) G_{d-1}(k)$ to obtain
\[
\dim V^{S G_d(n) G_{d-1}(k)} \le | S G_d(k) : S G_d(n) G_{d-1}(k) | \dim V^{S G_d(k)};
\]
note that $S G_d(n) G_{d-1}(k)$ is in fact a group, because $S$ and $G_{d-1}(k)$ commute and they both normalize $G_d(n)$.  Combining this with
\[
| S G_d(k) : S G_d(n) G_{d-1}(k) | = p^{\dim( G_d / S G_{d-1} )(n-k)} = p^{(2d-2)(n-k)}
\]
and the bound for $\dim V^{S G_d(k)}$ from Proposition \ref{Sinvariants} gives
\[
\dim V^{S G_d(n) G_{d-1}(k)} \le p^{(2d-2)(n-k)} C \kappa 10^k R^{2k/3} p^{(d^2-4)k/3} \le C \kappa 10^n p^{(2d-2)n} R^{2k/3} p^{(d^2 - 6d+2)k/3}.
\]
This implies that we may invoke the induction hypothesis for the representation of $G_{d-1}$ on $V_0 = V^{S G_d(n)}$ with data $C_0 = C \kappa 10^n p^{(2d-2)n}$, $R_0 = R^{2/3} p^{(d^2 - 6d+2)/3}$, and $N_0 = n$.  (It may be checked that $R_0 < p^{(d-1)^2 - 1}$.)  If $T_0$ is the diagonal subgroup of $G_{d-1}$, this gives

\begin{align*}
\dim V^{T G_d(n)} & =  \dim V_0^{T_0 G_{d-1}(n)} \\
& \le C_0 \kappa 10^{(d-2)n} R_0^{(2/3)^{d-2} n} p^{\sigma(d-1) n} \\
& = C \kappa 10^{(d-1) n} p^{(2d-2)n} R_0^{(2/3)^{d-2} n} p^{\sigma(d-1) n} \\
& = C \kappa 10^{(d-1) n} p^{(2d-2)n} R^{(2/3)^{d-1} n} p^{(2/3)^{d-2} n (d^2 - 6d+2)/3} p^{\sigma(d-1) n}.
\end{align*}
The proposition now follows after checking that the exponent of $p$ satisfies
\[
(2d-2) + (2/3)^{d-2} (d^2 - 6d+2)/3 + \sigma(d-1) = \sigma(d).
\]
\end{proof}

Finally, we deduce Theorem \ref{Tgen} from Proposition \ref{Tind}.  We are given that the conditions of the proposition hold with $R = p^{d^2-2}$, some $C > 0$, and any $N$, and it may be checked that in this case the proposition gives exactly the conclusion of Theorem \ref{Tgen}.

\end{document}